\documentclass[12pt,reqno]{amsart}
\usepackage{hyperref}
\usepackage{amsmath,amssymb}
\usepackage[english]{babel}
\usepackage{caption}
\usepackage{amssymb,amsmath,euscript,enumerate,tikz}
\usepackage[margin=1in]{geometry}
\usepackage{graphicx, caption}
\usepackage{float}
\usepackage{pgf,tikz}
\usepackage{mathrsfs}
\usepackage{upgreek}
\usepackage{mathtools}
\usepackage[utf8]{inputenc}
\usepackage{xcolor}
\usepackage{tikz-cd}

\newtheorem{theorem}{Theorem}[section]
\newtheorem{lemma}[theorem]{Lemma}
\newtheorem{proposition}[theorem]{Proposition}

\newtheorem{corollary}[theorem]{Corollary}

\theoremstyle{definition}

\newtheorem{example}[theorem]{Example}

\theoremstyle{remark}
\newcommand\cdeg{\operatorname{cdeg}}
\newcommand \pdeg{\operatorname{pdeg}}
\newtheorem{remark}[theorem]{Remark}

\usepackage{color,soul}

\newcommand{\pd}{\operatorname{pd}}
\newcommand{\pv}{\operatorname{pv}}
\newcommand{\iv}{\operatorname{iv}}

\newcommand{\T}{\operatorname{Tor}}

\newcommand{\cl}{\mathrm{cl}}
\newcommand{\reg}{\mathrm{reg}}

\begin{document}

	\title[Binomial Edge Ideals of Generalized block graphs]{Binomial Edge Ideals of Generalized block graphs}
	\author{Arvind Kumar}
	\email{arvkumar11@gmail.com}
	
	\address{Department of Mathematics, Indian Institute of Technology
        Madras, Chennai, INDIA - 600036}

	\begin{abstract}
		We classify generalized block graphs whose binomial edge ideals admit a unique extremal Betti number. 
		We prove that the Castelnuovo-Mumford regularity of binomial edge ideals of generalized block graphs is 
		bounded below by $m(G)+1$, where $m(G)$ is the number of minimal cut sets of the graph  $G$ and obtain
		an improved upper bound for the regularity in terms of the number of maximal cliques and
		pendant vertices of $G$.
	\end{abstract}
	
	
	\maketitle
\section{Introduction}
Let $R = K[x_1,\ldots,x_m]$ be a standard graded polynomial ring over an arbitrary
field $K$, and $M$ be a finitely generated graded  $R$-module. 
Let
\[
0 \longrightarrow \bigoplus_{j \in \mathbb{Z}} R(-j)^{\beta_{p,j}^R(M)} 
\overset{\phi_{p}}{\longrightarrow} \cdots \overset{\phi_1}{\longrightarrow} \bigoplus_{j \in \mathbb{Z}} R(-j)^{\beta_{0,j}^R(M)} 
\overset{\phi_0}{\longrightarrow} M\longrightarrow 0
\]
be the minimal graded free resolution of $M$, where $p \leq m$ and
$R(-j)$ is the free $R$-module of rank $1$ generated in degree $j$.
The number $\beta_{i,j}^R(M)$ is  called the 
$(i,j)$-{\it th graded Betti number} of $M$. The
projective dimension  and Castelnuovo-Mumford regularity (henceforth called regularity) are two invariants associated with
$M$ that can be read off from the minimal graded free resolution of $M$. The 
{\it regularity} of $M$, denoted by $\reg(M)$, is defined as 
\[
\reg(M):=\max \{j-i \mid \beta_{i,j}^R(M) \neq 0\},
\] and the {\it projective dimension} of $M$, denoted by $\pd_R(M)$, is defined as \[\pd_R(M):=\max\{i : \beta_{i,j}^R(M) \neq 0\}.\]
A Betti number $\beta_{i,j}^R(M)\neq 0$ is called an {\it extremal Betti number} if $\beta_{r,s}^R(M)=0$ 
for all pairs $(r,s)\neq (i,j)$ with $r\geq i$, and $s\geq j$. Observe that $M$ admits   a unique extremal Betti number if and only if 
$\beta_{p,p+r}^R(M)\neq 0$, where $p =\pd_R(M)$ and $r=\reg(M)$. For a graded $R$-module $M$,  we 
denote the {\it Betti polynomial} of $M$ by $$B_M(s,t)=\sum_{i,j}\beta_{i,j}^R(M)s^it^j.$$

Herzog et al. in \cite{HH1}
and independently Ohtani in \cite{oh} introduced the notion of
binomial edge ideal corresponding to a finite simple graph. Let $G$ be a  simple graph on $[n]$.   Let
$S=K[x_1, \ldots, x_n,y_1, \ldots, y_n]$, where $K$ is a field.  The
{\it binomial edge ideal} of $G$ is  $J_G =(x_i y_j - x_j y_i :
\{i,j\} \in E(G), \; i <j)$. Researchers have found exact formulas or bounds for algebraic invariants of $J_G$,  such as codimension, depth, Betti numbers and regularity, in terms of combinatorial invariants of the underlying graph $G$, see e.g., \cite{josep,bn,dav,her1, KM3, AR1, MM, KM1, KM2}. The study of regularity and Betti numbers of homogeneous ideals has attracted a lot of attention in the recent past due to its algebraic and geometric importance.  In \cite{MM}, Matsuda and Murai proved that if $G$ is a graph on the vertex
set $[n]$, then  $\ell(G) \leq \reg(S/J_G) \leq n-1$, where $\ell(G)$ is the length of a
longest induced path in $G$. In particular, they
conjectured, and later proved by Kiani and Saaedi Madani \cite{KM3}, that $\reg(S/J_G) = n-1$
if and only if $G$ is a path on $n$ vertices. Another conjectured upper bound for $\reg(S/J_G)$ is 
given by  $\cl(G)$, the number of maximal cliques of $G$ \cite{KM2}. The latter conjecture has been recently 
proved for chordal graphs by Rouzbahani
Malayeri et al. \cite{MKM2018}, extending results of 
\cite{EZ, JACM, KM1}, and for some classes of non-chordal graphs in \cite{AR2}.

An interesting class of chordal graphs consists of {\it block graphs}, which are connected graphs whose {\it blocks} (i.e., maximal subgraphs that cannot be disconnected by removing a vertex) are cliques. Recently, in \cite{her2},  Herzog and Rinaldo improved the lower bound for the regularity of binomial edge ideals of block graphs and classified block graphs whose binomial edge ideal admits a unique extremal Betti number. In this article, we extend these results to the class of generalized
block graphs that contains block graphs. We also obtain improved lower and upper bounds for the regularity of binomial edge ideals of these graphs.

The article is organized as follows. In the second section, we recall some results on graphs and binomial edge ideals.  In the third section, 
we characterize generalized block graphs whose binomial edge ideal admits a unique extremal Betti number (Theorem \ref{unique}). In particular, 
we prove  that $\beta_{p(G),p(G)+m(G)+1}^S(S/J_G)$ is an extremal Betti number of $S/J_G$ if $G$ is a connected
generalized block graph (Theorem \ref{main}), where $p(G) =\pd_S(S/J_G)$ and $m(G)$ is the number of
minimal cut sets of $G$. As a consequence, we obtain $\reg(S/J_G) \geq m(G)+1$ (Corollary \ref{least-reg}).
In the fourth section, we obtain improved upper bound for the regularity of binomial edge ideals of generalized
block graphs in terms of the number of maximal cliques and pendant vertices of $G$ (Theorem \ref{chordal-improved}). 

\section{Preliminaries}
In this section, we recall some notation and terminology
from graph theory, and some important results about binomial edge ideals.

Let $G$  be a  finite simple graph with vertex set $V(G)$ and edge set
$E(G)$. For  $A \subseteq V(G)$, $G[A]$ denotes the {\it induced subgraph}
of $G$ on the vertex set $A$, that is the subgraph with edge set $E(G[A])=\{\{i,j\} \in E(G) : i, j \in A\}$. 
For a vertex $v$, $G \setminus v$ denotes the  induced subgraph of $G$
on the vertex set $V(G) \setminus \{v\}$. A vertex $v \in V(G)$ is
said to be a \textit{cut vertex} if $G \setminus v$ has  more
connected components than $G$. For $T \subset [n]$, let $\overline{T}=[n] \setminus T$,
$c_G(T)$ be the number of connected components of $G[\overline{T}]$ and $c_G$ the number of connected
components of $G$. We say that a subset $T \subset [n]$ is a {\it cut set} of $G$ if $c_G (T ) > c_G$.
A cut set of $G$ is said to be a \textit{minimal cut set} if it is minimal under inclusion.
A subset $U$ of $V(G)$ is said to be a 
\textit{clique} if $G[U]$ is a complete graph. 

Let $\Delta$ be a simplicial complex on the vertex set $[n]$ and $\mathcal{F}(\Delta)$ be the set of its facets. 
A facet $F$ of $\Delta$ is called a {\em leaf} if either $F$ is the only facet of $\Delta$ or else there exists a 
facet $G$, called a {\it branch} of $F$, such that for each facet $H$ of $\Delta$ with $H\neq F$, $H\cap F\subseteq G\cap F$.  

The simplicial complex $\Delta$ is called a {\em quasi-forest} if its facets can be ordered as
$F_1,\ldots,F_s$ such that for all $i>1$, the facet $F_i$ is a leaf of the simplicial complex with facets
$F_1,\ldots, F_{i-1}$. Such an order of the facets is called a {\em leaf order}. The simplicial complex whose facets
are the maximal cliques of a graph $G$ is called the {\it clique complex} of $G$ and denoted by $\Delta(G)$.
By \cite[Theorem 9.2.12]{MI}, $G$ is chordal if and only if $\Delta(G)$ is a quasi-forest.

A vertex
$v$ of $G$ is said to be a {\em free vertex} if $v$ belongs to exactly one maximal clique of $G$.
A vertex $v$ is said to be an \textit{internal vertex} of $G$ if it is not a free vertex. The set $N_G(v) = \{u \in V(G) ~
: ~ \{u,v\} \in E(G)\}$ is called the {\it neighborhood} of $v$, and  $G_v$ denotes the
graph on the vertex set $V(G)$ and edge set $E(G_v) =E(G) \cup \{
\{u,w\}: u,w \in N_G(v)\}$. Observe that, if $v$ is a free vertex, then $G_v=G$. 

Let $G_1,\ldots,G_{c_{G}(T)}$ be the connected 
components of $G[\overline{T}]$. For each $i$, let $\widetilde{G_i}$ denote the complete graph on $V(G_i)$ and
$P_T(G) = (\underset{i\in T} \cup \{x_i,y_i\}, J_{\widetilde{G_1}},\ldots, J_{\widetilde{G_{c_G(T)}}}).$
In \cite{HH1}, it was shown by Herzog et al.  that $J_G =  \underset{T \subseteq [n]}\cap P_T(G)$.
For each $i \in T$, if $i$ is a cut vertex of the graph $G[\overline{T} \cup \{i\}]$,
then we say that $T$ has the {\it cut point property}. Set $\mathcal{C}(G) =\{\emptyset \}
\cup \{ T: T \; \text{has the cut point property} \}$. It follows from  \cite[Corollary 3.9]{HH1} 
that $T \in \mathcal{C}(G)$ if and only if $P_T(G)$ is a minimal prime of $J_G$. Hence, by \cite[Theorem 3.2]{HH1} 
and \cite[Corollary 3.9]{HH1}, we have $J_G =  \underset{T \in \mathcal{C}(G)} \cap P_T(G)$.

\section{Extremal Betti number of generalized block graphs}
In this section, we study the extremal Betti number $\beta_{p(G),p(G)+j}^S(S/J_G)$ of binomial edge ideals 
of generalized block graphs, where $p(G) =\pd_S(S/J_G)$. A maximal connected subgraph of $G$ with no cut vertex is called a \textit{block}. 
A graph $G$ is called a  \textit{block graph} if each block of $G$ is a  clique.  In other words,  a block graph is a chordal graph such that every pair of blocks of $G$ intersects in at most one vertex.  Block graphs 
were extensively studied by many authors, see \cite{her1}, \cite{EZ}, \cite{her2}, \cite{JNR}.

Generalized block graphs are the 
generalization of block graphs and were introduced in \cite{KM6}.   A   chordal 
graph  $G$ is said to be a \textit{generalized block graph} if
$F_i,F_j,F_k\in  \mathcal{F}(\Delta(G))$ such that  $F_i\cap F_j\cap F_k\neq \emptyset$, then $F_i\cap F_j=F_i\cap F_k=F_j\cap F_k$.
One could see that all block graphs are  generalized block graphs.  By definition of  generalized block graph, it is clear that a 
subset $A$ of  vertices of   $G$ is a minimal cut set if and only if there exist $F_{t_1},\ldots,F_{t_q}\in \mathcal{F}(\Delta(G))$
such that $\bigcap_{j=1}^{q}F_{t_j}=A$, and for all other facets $F$ of $\Delta(G)$, $F\cap A=\emptyset$. Note that if $A$ is a minimal cut set,
then $A$ is a clique. For a minimal cut set $A$, we denote by $G_A$ the graph obtained from $G$ by 
replacing the cliques $F_{t_1},\ldots,F_{t_q}$ with  the clique on the vertex set $\underset{ j \in [q]
} \cup F_{t_j}$.

\begin{lemma}\label{cut-lemma1}
	Let $G$ be a graph on the vertex set $[n]$ and $A$ be a minimal cut set of $G$. 
	Then $A \in \mathcal{C}(G)$. Moreover, if $G$ is a generalized block graph, then for 
	every  $T \in \mathcal{C}(G)$, either $A \subseteq T$ or $ A \cap T = \emptyset$.
\end{lemma}
\begin{proof}
	For the first claim, suppose that $A \notin \mathcal{C}(G)$. Then there exists $v \in A$ such 
	that $c_G (A \setminus \{v\}) = c_G (A)$. Since
	$A$ is a minimal cut set, we know that $c_G (A) > c_G$, and hence, $c_G (A \setminus \{v\}) > c_G$.
	This means that $A \setminus \{v\}$ is a cut set and is properly contained in $A$, against the minimality
	of $A$. Thus, $A \in \mathcal{C}(G)$. The second claim clearly holds when $|A| = 1$. Let $|A| \geq 2$ and 
	$T \in \mathcal{C}(G)$. If $T \cap A =\emptyset$, then there is nothing to prove. Assume that 
	$T \cap A \neq \emptyset$ and let $v \in T \cap A$. Then $v$ is a cut vertex of $G[\overline{T}\cup \{v\}]$.
	Suppose that there exists $ w \in A \setminus T$. We want to show that $N_G (v) = N_G (w)$. Suppose that there 
	exists $u \in  N_G (w) \setminus N_G (v)$ and let $F$ be a facet
	of $\Delta(G)$ containing $w$ and $u$. Since $A$ is a cut set, there are at least two facets $F_1 , F_2$ containing
	$v, w$ but not $u$. Then, we have $F_1 \cap F_2 \cap F \neq \emptyset$,  $v \in F_1 \cap F_2$ but
	$v \notin F_1 \cap F$ and $v \notin F_2 \cap F$,
	against the fact that G is a generalized block graph. Thus, $N_G(w)\subset N_G(v)$. Similarily, $N_G(v)\subset N_G(w)$, and hence,  $N_G(v)=N_G(w)$. Consequently, $v$ is not a cut vertex of   $G[\overline{T}\cup \{v\}]$,
	which is a contradiction. Hence, $A \subseteq T.$    
\end{proof} 
We now give an example of a chordal graph $G$ that is not a generalized block graph 
for which the second claim of Lemma \ref{cut-lemma1}, does not hold.
\begin{example}
	Let $G$ be a graph as shown in Fig. \ref{cho0}. Then, it can be seen that $G$ is
	a chordal graph that is not a generalized block graph. The sets  $A= \{2,3\}$ and $ B=\{3,4\}$ are
	minimal cut sets of $G$, thus  $A,B \in \mathcal{C}(G)$, and $A\cap B \neq \emptyset$. 
	However, $A\not\subseteq B$ and $B \not\subseteq A$.    
	\begin{figure}[h]
		\begin{center}
			\begin{tikzpicture}[scale=.5]
			\draw  (0,2)-- (2,0);
			\draw  (2,0)-- (4,2);
			\draw  (4,2)-- (4,0);
			\draw  (4,0)-- (2,0);
			\draw  (2,0)-- (0,0);
			\draw  (0,0)-- (0,2);
			\draw  (0,2)-- (4,2);
			\begin{scriptsize}
			\draw (0,2) circle (1.5pt);
			\draw (0,2.3) node {$2$};
			\draw (2,0) circle (1.5pt);
			\draw (2.0,-0.3) node {$3$};
			\draw (4,2) circle (1.5pt);
			\draw (4,2.3) node {$4$};
			\draw (4,0) circle (1.5pt);
			\draw (4,-0.3) node {$5$};
			\draw (0,0) circle (1.5pt);
			\draw (0,-0.3) node {$1$};
			\end{scriptsize}
			\end{tikzpicture}
			\caption{}\label{cho0}
			
		\end{center}
	\end{figure}    
\end{example}
Let $G$ be a  generalized graph on $[n]$.  Let $A$ be a minimal cut set of $G$. Set
\begin{equation}
Q_1=\bigcap_{\substack{
		T\subseteq [n] \\
		A\cap T = \emptyset
	}}
	P_T(G)~~,~~Q_2=\bigcap_{\substack{
			T\subseteq [n] \\
			A\subseteq T
		}}
		P_T(G).
		\nonumber
		\end{equation}
		By \cite[Theorem 3.2, Corollary 3.9]{HH1}, $J_{G_A} =\bigcap_{T\subset [n]}P_T(G_A)=\bigcap_{T \in \mathcal{C}(G_A)}P_T(G_A).$    
		It follows from \cite[Proposition 2.1]{Rauf} that $T \in \mathcal{C}(G_A)$ if and only if 
		$A \cap T =\emptyset$ and $T\in \mathcal{C}(G_A[\overline{A}])$. If $A \cap T=\emptyset $, then  $P_T(G)=P_T(G_A)$. Consequently, 
		$Q_1= J_{G_A}$. Note that 
		$$Q_2 =(x_i,y_i : i \in A)+\bigcap_{T\setminus A \subset [n]\setminus A}P_{T\setminus A}(G[\overline{A}])=(x_i,y_i : i \in A) 
		+J_{G[\overline{A}]},$$ where the last equality follows from \cite[Theorem 3.2]{HH1}. 
		Thus, $$Q_1+Q_2 = (x_i,y_i : i \in A) + J_{G_A[\overline{A}]}.$$            
		By virtue of Lemma \ref{cut-lemma1}, $J_G = Q_1 \cap Q_2$. This gives us the following short exact sequence,
		\begin{equation}\label{ses1}
		0  \longrightarrow  \dfrac{S}{J_{G} } \longrightarrow
		\dfrac{S}{Q_1} \oplus \dfrac{S}{Q_2} 
		\longrightarrow \dfrac{S}{Q_1+Q_2 } 
		\longrightarrow 0.
		\end{equation}
		The following example illustrates that, in general, $Q_1 \neq J_{G_A}$ for a minimal set $A$ of $G$.
		\begin{example}
			Let $G$ be a graph as shown in Fig. \ref{cho0}. Then, $T \in \mathcal{C}(G)$ if and only if $T\in \{ \emptyset, \{2,3\}, \{3,4\} \}$. Set  $A= \{2,3\}$ and $B=\{3,4\}$. Note that $A$ is a minimal cut set of $G$ and $G_A=G\cup \{1,4\}$. Let $T \subseteq [5]$ such that $A \cap T =\emptyset$. Then, $T \subseteq \{1,4,5\}$, and hence, $P_{\emptyset}(G) \subseteq P_T(G)$. Thus, $$Q_1=\bigcap_{T \subseteq \{1,4,5\}} P_T(G)=P_{\emptyset}(G)=J_{K_5}\neq J_{G_A}.$$ However, \begin{align*}
			Q_2&=\bigcap_{T\subseteq [5], \;A \subset T}P_T(G)=(x_2,y_2,x_3,y_3)+ \bigcap_{T\subseteq \{1,4,5\}}P_T(G[\{1,4,5\}])\\&=(x_2,y_2,x_3,y_3, x_4y_5-x_5y_4)=P_{A}(G).
			\end{align*}
			Since $J_G=P_{\emptyset}(G) \cap P_A(G) \cap P_B(G)$, we have $J_G \neq  Q_1 \cap Q_2.$
		\end{example}
		Let $p(G)$ denote the projective dimension of $S/J_G$. Then $\pd_S(S/Q_1) =p(G_A)$. Let $T=K[x_1,\ldots,x_n]$.
		Let $0 < m< n$, $I\subset R= K[x_1,\ldots,x_m]$ and $J\subset R'=
		K[x_{m+1},\ldots,x_n]$ be  homogeneous  ideals. Then,   
		the minimal graded free resolution of ${T}/{(I + J)}$
		is  the tensor product of the
		minimal free resolutions of ${R}/{I}$ and  ${R'}/{J}$. For $A \subset [n]$, set  $S_A=K[x_i,y_i : \;  i \notin A]$.
		Hence, $\pd_S(S/Q_2) = 2|A|+\pd_{S_A}(S_A/J_{G[\overline{A}]})=2|A|+p(G[\overline{A}])$
		and $\pd_S(S/(Q_1+Q_2)) =2|A|+\pd(S_A/J_{G_A[\overline{A}]})= 2|A|+p(G_A[\overline{A}])$. 
		
		In \cite{KM6}, Kiani and Saeedi Madani  obtained the  depth of binomial edge ideals of generalized block graphs, and
		hence, their projective dimension by the Auslander-Buchsbaum formula. Recall that the {\it clique number} of a
		graph G, denoted $\omega(G)$, is the maximum size of the maximal cliques of $G$. Let $G$ be
		a generalized block graph on $[n]$. For each $i=1,\ldots,\omega(G)-1$, we set 
		$$\mathcal{A}_{i}(G)=\{A\subseteq [n] : |A|=i,A~\mathrm{is~a~minimal~cut~set~of~}G\}$$
		and  $a_i(G)=|\mathcal{A}_{i}(G)|$. Observe that a generalized block graph $G$ is a block graph if and only if $a_i(G)=0$,
		for all $i>1$. Let $m(G)$ denote the number of minimal cut sets of $G$. Then, $m(G)= \sum_{i=1}^{\omega(G)-1} a_i(G)$.
		
		By \cite[Theorem 3.2]{KM6} and the Auslander-Buchsbaum formula, it follows that:
		\begin{theorem}\label{pd-gb}    Let $G$ be a generalized block graph on $[n]$. Then,
			$$p(G)=n-c_G+\sum_{i=2}^{\omega(G)-1}(i-1)a_i(G),$$ where $c_G$ is the number of  connected components of $G$.
		\end{theorem}

		We recall the notion of decomposability from \cite[Section 2]{Rauf} and \cite{cactus}. 
		A graph $G$ is called \textit{decomposable} if there exist subgraphs $G_1$ and $ G_2$ such that  
		$G= G_1 \cup G_2$, $V(G_1)\cap V(G_2)=\{v\}$ and  $v$ is a 
		free vertex of both $G_1$ and $G_2$.

		A graph $G$ is called \textit{indecomposable} if it is not decomposable. Up to ordering, $G$
		has  a unique decomposition into indecomposable subgraphs, i.e., there exist 
		$G_1,\ldots,G_r$ indecomposable induced subgraphs of $G$ with 
		$G=G_1\cup \cdots \cup G_r$ such that for each $i \neq j$, either 
		$V(G_i) \cap V(G_j) = \emptyset$ or $V(G_i) \cap V(G_j) =\{v\}$ and $v$ is a  free vertex 
		of both $G_i$ and $G_j$.
		
		In \cite[Proposition 3]{her2}, Herzog and Rinaldo proved that:
		\begin{proposition}\cite[Proposition 1.3]{her2}\label{decomposable}
			Let $G=G_1\cup G_2$ be a decomposable graph. Let
			$S_i =K[x_j,y_j:j\in V(G_i)]$, for $i=1,2$. Then,
			\[
			B_{{S}/{J_G}}(s,t)= B_{{S_1}/{J_{G_1}}}(s,t)
			B_{{S_2}/{J_{G_2}}}(s,t).
			\]
		\end{proposition}
		It follows from Proposition  \ref{decomposable} that if $G = G_1 \cup \cdots \cup G_r$
		is a decomposition of $G$ into indecomposable graphs, then
		$\reg(S/J_G) = \underset{i \in [r]} \sum \reg(S_i/J_{G_i})$ and  $p(G) = \underset{ i \in [r]} \sum p({G_i})$.
		Also, if for each $i$, $\beta^{S_i}_{p(G_i),p(G_i)+j_i}(S_i/J_{G_i})$ is an extremal Betti
		number of $S_i/J_{G_i}$, then $$\beta^S_{p(G),p(G)+j}(S/J_G)=\underset{i\in [r]} \prod \beta^{S_i}_{p(G_i),p(G_i)+j_i}(S_i/J_{G_i})$$ is an
		extremal Betti number of $S/J_G$, where $j=j_1+\cdots+j_r$. Therefore,  it is enough to find the position of the extremal 
		Betti number $\beta_{p(G),p(G)+i}^S(S/J_G)$ for  indecomposable graphs.
		\begin{lemma}\label{gen-lem}
			Let $G$ be a connected indecomposable generalized block graph and let $F_1,\ldots,F_r$ be a
			leaf order of $\mathcal{F}(\Delta(G))$. Denote by $F_{t_1},\ldots,F_{t_q}$ all the branches 
			of the leaf $F_r$. Set $A=F_r \cap \bigcap_{i=1}^qF_{t_i}$ and $\alpha=|A|$. Then, 
			\begin{enumerate}
				\item[(a)] the graphs $G_A, \; G_A[\overline{A}]$ and $G[\overline{A}]$ are generalized block graphs.
				\item[(b)] for $i \neq \alpha$, $a_i(G_A)=a_i(G)$ and $a_{\alpha}(G_A)=a_{\alpha}(G)-1$. 
				In particular, $$m(G_A)=m(G)-1\; \text{ and }\; p(G_A)=p(G)-\alpha+1. $$ 
				\item[(c)] for $i \neq \alpha$, $a_i(G_A[\overline{A}])=a_i(G)$ and $a_{\alpha}(G_A[\overline{A}])=a_{\alpha}(G)-1$.
				In particular, $$m(G_A[\overline{A}])=m(G)-1\; \text{ and }\; p(G_A[\overline{A}])=p(G)-2\alpha+1. $$ 
				\item[(d)] for $i \neq \alpha$, $a_i(G[\overline{A}])\leq a_i(G)$ and $a_{\alpha}(G[\overline{A}])\leq a_{\alpha}(G)-1$.
				In particular, $$m(G[\overline{A}])\leq m(G)-1\; \text{ and }\; p(G[\overline{A}])=p(G)-2\alpha-q+1. $$     
			\end{enumerate}
		\end{lemma}
		\begin{proof}
			(a) This easily follows by the fact that $G$ is a generalized block graph.
			\par (b) Notice that for $i \neq \alpha, \mathcal{A}_i(G_A)=\mathcal{A}_i(G)$ and 
			$\mathcal{A}_{\alpha}(G_A)=\mathcal{A}_{\alpha}(G)\setminus \{A\}$. Thus, by Proposition
			\ref{pd-gb}, $p(G_A)=n-1+\sum_{i=2}^{\omega(G_A)-1}(i-1)a_i(G_A)=p(G)-\alpha+1$ and  $m(G_A)=m(G)-1$.
			\par (c)  Notice that for $i \neq \alpha, \mathcal{A}_i(G_A[\overline{A}])=\mathcal{A}_i(G)$ and 
			$\mathcal{A}_{\alpha}(G_A[\overline{A}])=\mathcal{A}_{\alpha}(G)\setminus \{A\}$. Thus, 
			by Proposition \ref{pd-gb}, 
			$p(G_A[\overline{A}])=(n-\alpha)-1+\sum_{i=2}^{\omega(G_A[\overline{A}])-1}(i-1)a_i(G_A[\overline{A}])=p(G)-2\alpha+1$ 
			and $m(G_A[\overline{A}])= \sum_{i=1}^{\omega(G_A[\overline{A}])-1}a_i(G_A[\overline{A}])=m(G)-1$.
			\par (d) Let $B$ be a minimal cut set of $G[\overline{A}]$. Since $G[\overline{A}]$ is an induced 
			subgraph of $G$ and $B\cap A=\emptyset$, $B$ is a minimal cut set of $G$. Therefore, for $i \neq \alpha$, 
			$\mathcal{A}_i(G[\overline{A}])\subseteq \mathcal{A}_i(G)$ and 
			$\mathcal{A}_{\alpha}(G[\overline{A}])\subseteq \mathcal{A}_{\alpha}(G)\setminus \{A\}$. Thus,
			$m(G[\overline{A}]) = \sum_{i=1}^{\omega(G[\overline{A}])-1}a_i(G[\overline{A}])\leq m(G)-1$ and 
			by Proposition \ref{pd-gb}, $p(G[\overline{A}])=(n-\alpha)-(q+1)+\sum_{i=2}^{\omega(G[\overline{A}])-1}(i-1)a_i(G[\overline{A}])\leq
			p(G)-2\alpha-q+1$.
		\end{proof}
		Recall that a vertex $v$ is said to be an {\it internal vertex} of $G$  if it is not a free vertex. 
		For $v \in V(G)$, let $\cdeg_G(v)$ denote the number of maximal cliques of $G$ which contains $v$. 
		The number of free vertices of $G$ is denoted by $f(G)$.
		
		\begin{theorem}
			\label{main}
			Let $G$ be a connected indecomposable   generalized block graph on the vertex set $[n]$.
			Then,  $\beta_{p(G),p(G)+m(G)+1}^S(S/J_G)$ is an  extremal Betti number  of $S/J_G$. Moreover, 
			if $G$ is a complete graph or for every internal vertex
			$v$, $\cdeg_G(v) >2$, then $\beta_{p(G),p(G)+m(G)+1}^S(S/J_G)=f(G)-1$.
		\end{theorem}
		\begin{proof}
			We prove this assertion by induction on $m(G)$. If $m(G)=0$, then $G$ is a complete graph.
			Therefore, the claim follows by the Eagon-Northcott resolution \cite{EN}.
			Assume that  $m(G)>0$. Since  $G$ is a chordal graph, by \cite[Theorem 9.2.12]{MI},
			$\Delta(G)$ is a quasi-forest. Let $F_1,\ldots,F_r$ be a leaf order of $\mathcal{F}(\Delta(G))$.
			Let $F_{t_1},\ldots,F_{t_q}$ be all the branches of the leaf $F_r$. Note that $q\geq 1$. Since $G$ is 
			a generalized block graph,  $F_r \cap F_{t_i}=F_{t_j}\cap F_{t_k}$ for every pair of $i,j,k \in [q]$ with
			$j \neq k$ and for all $l\neq t_1,\ldots,t_q$, 
			$F_r\cap F_l=\emptyset$. Let $A =F_r \cap F_{t_1}= \cap_{i=1}^q F_{t_i} \cap F_r$  and  $\alpha =|A|$.
			Since  $A$ is a minimal cut set, by the discussion after Lemma \ref{cut-lemma1}, 
			$J_G =Q_1 \cap Q_2$, where $Q_1= J_{G_A}$ and $Q_2 =(x_i,y_i : i \in A) +J_{G[\overline{A}]}$.
			
			By Lemma \ref{gen-lem}, $G_A,\; G_A[\overline{A}]$ and $G[\overline{A}]$ are generalized block graphs.
			We have the following  cases:
			
			\textbf{Case (1):} If $\alpha =1$, then it follows from Theorem \ref{pd-gb}  
			and Lemma \ref{gen-lem} that $p(G_A)=p(G)$, $p(G_A[\overline{A}])=p(G)-1$ and 
			$p(G[\overline{A}]) \leq p(G)-q-1$(where  $G[\overline{A}]$ has $q+1$ connected components). 
			Note that $G[\overline{A}]$ is not necessarily indecomposable, but we can split it into smaller indecomposable graphs. 
			Since  $G$ is an indecomposable graph, $q \geq 2$, and hence, $\pd_S(S/Q_2) =2 +p(G[\overline{A}])\leq p(G)-1$.
			Therefore,  $$\T_{i}^S \left(\frac{S}{Q_2},K\right)=0, \text{ for } i \geq p(G).$$ Thus, for each $j\geq 0$,
			the exact sequence \eqref{ses1} 
			yields the long exact sequence of Tor sequence:
			\begin{equation}\label{Tor}
			\begin{split}
			0\rightarrow \T_{p(G)+1,p(G)+j}^S\left(\frac{S}{Q_1+Q_2},K\right)&\rightarrow
			\T_{p(G),p(G)+j}^S\left(\frac{S}{J_{G}},K\right) \rightarrow \\& \rightarrow 
			\T_{p(G),p(G)+j}^S\left(\frac{S}{J_{G_A}},K\right)\rightarrow \ldots
			\end{split}
			\end{equation}
			Since $Q_1+Q_2=(x_i,y_i : i\in A)+J_{G_A[\overline{A}]}$, we have that
			\begin{equation}\label{TorI1}
			\T^S_{p(G)+1,p(G)+j}\left(\frac{S}{Q_1+Q_2},K \right) \cong 
			\T^{S_A}_{p(G_A[\overline{A}]),p(G_A[\overline{A}])+(j-1)}\left(\frac{S_A}{J_{G_A[\overline{A}]}},K\right)
			\end{equation} where $S_A=K[x_i,y_i : i \notin A]$.
			It follows from  induction that
			\begin{equation}\label{Tor4}
			\T_{p(G_A[\overline{A}]),p(G_A[\overline{A}])+(j-1)}^{S_A}\left(\frac{S}{J_{G_A[\overline{A}]}},K\right)
			=0\quad \text{for}\quad j>m(G_A[\overline{A}])+2=m(G)+1,
			\end{equation}
			and
			\[
			\T_{p(G),p(G)+j}^S\left(\frac{S}{J_{G_A}},K\right)=0\quad \text{for}\quad j>m(G_A)+1=m(G).
			\]
			Now, \eqref{Tor}, \eqref{TorI1} and \eqref{Tor4} imply that
			\begin{equation}\label{Tor2}
			\T_{p(G),p(G)+j}^S\left(\frac{S}{J_{G}},K\right)=0 \quad \text{for}\quad j>m(G)+1,
			\end{equation} 
			and
			\begin{equation}\label{TorI2}
			\T^{S_A}_{p(G_A[\overline{A}]),p(G_A[\overline{A}])+m(G_A[\overline{A}])+1}\left(\frac{S_A}{J_{G_A[\overline{A}]}},K\right)
			\cong \T^S_{p(G),p(G)+m(G)+1}\left(\frac{S}{J_{G}},K\right).
			\end{equation}
			By induction, $\beta^{S_A}_{p(G_A[\overline{A}]),p(G_A[\overline{A}])+m(G_A[\overline{A}])+1}(S_A/J_{G_A[\overline{A}]})\neq 0$
			is an extremal Betti number. Now,   Eq.
			\eqref{TorI2} implies $$\beta^S_{p(G),p(G)+m(G)+1}(S/J_G) \neq 0,$$ and by Eq. \eqref{Tor2},
			we get that $\beta_{p(G),p(G)+m(G)+1}^S(S/J_G)$ is an extremal Betti number.
			
			\textbf{Case (2):} If $\alpha \geq 2$, then by virtue of Theorem \ref{pd-gb} and
			Lemma \ref{gen-lem}, $p(G_A)=p(G)-\alpha+1$, $p(G_A[\overline{A}])=p(G)-2\alpha+1$. Therefore, 
			\[
			\T_{i}^S\left(\frac{S}{Q_1},K\right)=
			\T_{i}^S\left(\frac{S}{J_{G_A}},K\right)=0, \text{ for } i \geq p(G).
			\] Note that $G[\overline{A}]$ has $q+1$ connected components. By  Theorem \ref{pd-gb}
			and Lemma \ref{gen-lem}, we have that $p(G[\overline{A}]) \leq p(G) -2\alpha -q+1$.
			Therefore, $\pd_S(S/Q_2) =2\alpha +p(G[\overline{A}])\leq p(G)-q+1$.
			Thus, for each $j\geq 0$, the exact sequence \eqref{ses1} yields the long exact sequence of Tor sequence: 
			\begin{equation}\label{Tor3}
			\begin{split}
			0\rightarrow \T_{p(G)+1,p(G)+j}^S\left(\frac{S}{Q_1+Q_2},K\right)&\rightarrow
			\T_{p(G),p(G)+j}^S\left(\frac{S}{J_{G}},K\right) \rightarrow\\& \rightarrow
			\T_{p(G),p(G)+j}^S\left(\frac{S}{Q_2},K\right)\rightarrow \ldots
			\end{split}            
			\end{equation}
			We now distinguish between two sub-cases. 
			
			{\bf Case (2.1):}
			If $\pd_S(S/Q_2) =2\alpha +p(G[\overline{A}])\leq p(G)-1$, then  $$\T_{p(G)}^S \left(\frac{S}{Q_2},K\right)=0.$$ 
			For each $j\geq 0$, \eqref{Tor3} yields that 
			\begin{equation}\label{TorIso1}
			\T^S_{p(G)+1,p(G)+j}\left(\frac{S}{Q_1+Q_2},K \right) \cong  \T^{S}_{p(G),p(G)+j}\left(\frac{S}{J_G},K\right).
			\end{equation}
			Now,  Eqs. \eqref{TorI1}, \eqref{Tor4} and \eqref{TorIso1} imply  
			\begin{equation}\label{Tor5}
			\T_{p(G),p(G)+j}^S\left(\frac{S}{J_{G}},K\right)=0\quad \text{for}\quad j>m(G)+1,
			\end{equation} and
			\begin{equation}\label{TorIso2}
			\T^{S_A}_{p(G_A[\overline{A}]),p(G_A[\overline{A}])+m(G_A[\overline{A}])+1}\left(\frac{S_A}{J_{G_A[\overline{A}]}},K\right)
			\cong \T^S_{p(G),p(G)+m(G)+1}\left(\frac{S}{J_{G}},K\right).
			\end{equation}
			By induction, $\beta^{S_A}_{p(G_A[\overline{A}]),p(G_A[\overline{A}])+m(G_A[\overline{A}])+1}(S_A/J_{G_A[\overline{A}]})\neq 0$ 
			is an extremal Betti number. 
			Hence, Eq.  \eqref{TorIso2} implies $$\beta^S_{p(G),p(G)+m(G)+1}(S/J_G) \neq 0,$$ and by Eq. \eqref{Tor5},
			we get that $\beta_{p(G),p(G)+m(G)+1}^S(S/J_G)$ is an extremal Betti number.

			{\bf Case (2.2):}
			If $\pd_S(S/Q_2)=p(G)$, then $q=1$. Let $H_1$ and $H_2$ be connected components
			of $G[\overline{A}]$. Then,  $m(G[\overline{A}])=m(H_1)+m(H_2)$. 
			For $i=1,2$, set $S_{H_i}=K[x_j,y_j : j \in V(H_i)]$. If $H_2$ is an isolated vertex, then
			$\pd_S(S/Q_2)=2\alpha +p(H_1)=2\alpha+p(G[\overline{A}])$ 
			and $m(H_1)=m(G[\overline{A}]) \leq m(G)-1$. 
			If $H_2$ is a non-trivial graph, then $\pd_S(S/Q_2)=2\alpha+p(H_1)+p(H_2)=2\alpha+p(G[\overline{A}])$ 
			and $m(H_1)+m(H_2)+2=m(G[\overline{A}])+2 \leq m(G)+1$.
			By induction, 
			\[
			\T_{p(G[\overline{A}]),p(G[\overline{A}])+j}^{S_A}\left(\frac{S}{J_{G[\overline{A}]}},K\right)=
			0\quad \text{for}\quad j>m(G)+1\geq m(G[\overline{A}])+2. 
			\] 
			Thus, $\T_{p(G),p(G)+j}^{S}\left(\frac{S}{Q_2},K\right)=0, \text{ for } j>m(G)+1.$
			Now,  Eqs. \eqref{TorI1}, \eqref{Tor4} and \eqref{Tor3} imply 
			\begin{equation}\label{Tor6}
			\T_{p(G),p(G)+j}^S\left(\frac{S}{J_{G}},K\right)=0\quad \text{for}\quad j>m(G)+1.
			\end{equation}
			By induction, $\beta^{S_A}_{p(G_A[\overline{A}]),p(G_A[\overline{A}])+m(G_A[\overline{A}])+1}(S_A/J_{G_A[\overline{A}]})\neq 0$.
			Consequently, by  Eq. \eqref{Tor3}, $$\beta^S_{p(G),p(G)+m(G)+1}(S/J_G) \neq 0,$$ and together with Eq. \eqref{Tor6}, we get
			that $\beta_{p(G),p(G)+m(G)+1}^S(S/J_G)$ is an extremal Betti number. 
			
			If $G$ is a complete graph, then $p(G)=n-1$ and $m(G)=0$. It follows from \cite[Corollary 4.3]{HKS}
			that $\beta_{n-1,n}^S(S/J_G)=n-1=f(G)-1$.
			We now assume that for every internal vertex $v$, $\cdeg_G(v)>2$. Therefore, $q\geq 2$, 
			and as before we conclude that 
			\[\T^{S_A}_{p(G_A[\overline{A}]),p(G_A[\overline{A}])+m(G_A[\overline{A}])+1}\left(\frac{S_A}{J_{G_A[\overline{A}]}},K\right)
			\cong \T^S_{p(G),p(G)+m(G)+1}\left(\frac{S}{J_{G}},K\right).\]
			Now, by induction, $\beta^{S_A}_{p(G_A[\overline{A}]),p(G_A[\overline{A}])+m(G_A[\overline{A}])+1}(S_A/J_{G_A[\overline{A}]})
			=f(G_A[\overline{A}])-1$. Since 
			$f(G)=f(G_A[\overline{A}])$, we conclude that $\beta^S_{p(G),p(G)+m(G)+1}(S/J_G) =f(G)-1.$
		\end{proof}
		
		\begin{corollary}\label{Lem:Extremal}
			Let $G$ be a generalized block graph for which $G=G_1\cup \cdots \cup G_{s}$ is the decomposition of
			$G$ into indecomposable graphs. 
			Then,   $\beta_{p(G),p(G)+m(G)+1}^S(S/J_G)$ is an extremal Betti number of $S/J_G$.
		\end{corollary}
		\begin{proof}
			Note that $m(G)=m(G_1)+\cdots + m(G_s)+s-1$. Now, the assertion follows from Proposition \ref{decomposable}
			and Theorem \ref{main}.
		\end{proof}    
		As of now, the only lower bound  known for regularity of binomial edge ideals of generalized block graphs
		is $\ell(G)$, which is a general lower bound given by 
		Matsuda and Murai. If $H$ is a longest induced path of a generalized block graph $G$, 
		then $\ell(G)=\ell(H)=m(H)+1 \leq m(G)+1$. Thus, as an immediate consequence of Theorem
		\ref{main}, we obtain an improved lower bound for the regularity of binomial edge ideals of generalized block graphs.
		\begin{corollary}\label{reg-lb}
			Let $G$ be a generalized block graph on $[n]$ with $c_G$ connected
			components. Then,  $\reg(S/J_G)\geq m(G)+c_G$.
		\end{corollary}
		\begin{proof}
			Let $G_1,\ldots,G_{c_G}$ be connected components of $G$. For $1 \leq i \leq c_G$, set  $S_i=K[x_j,y_j : j \in V(G_i)]$.
			Then, for $1\leq i \leq c_G$, by Corollary \ref{Lem:Extremal}, $\reg(S_i/J_{G_i}) \geq m(G_i)+1$. Note 
			that $m(G)=m(G_1)+\cdots+m(G_{c_G})$ and $S/J_G \simeq S_1/J_{G_1}\otimes \cdots \otimes S_{c_G}/J_{G_{c_G}}$. 
			Thus, the minimal graded free resolution of $S/J_G$ is the tensor product of the minimal free resolutions
			of $ S_1/J_{G_1}, \ldots,S_{c_G}/J_{G_{c_G}}$.  Hence, $\reg(S/J_G)=\sum_{i=1}^{c_G} \reg(S_i/J_{G_i}) \geq m(G)+c_G$.
		\end{proof}
		
		We now give an example of a connected chordal graph $G$ that is not a generalized block graph     
		for which $\reg(S/J_G) <m(G)+1$.     \begin{example}
			Let $G$ be a graph as shown in Fig. \ref{cho}. Then, it can be seen that $G$ is a chordal graph 
			that is not a generalized block graph. The minimal cut sets of $G$ are $\{2,3\}, \{2,5\}, \{3,5\}$. 
			Therefore, $m(G)=3$. Using Macaulay2 \cite{M2}, it can be seen that $\reg(S/J_G)=3<m(G)+1=4$.
			
			\begin{figure}[h]
				\begin{center}
					\begin{tikzpicture}[scale=.5]
					\draw  (0,2)-- (2,0);
					\draw  (2,0)-- (4,2);
					\draw  (4,2)-- (4,0);
					\draw  (4,0)-- (2,0);
					\draw  (2,0)-- (0,0);
					\draw  (0,0)-- (0,2);
					\draw  (0,2)-- (4,2);
					\draw  (4,2)-- (2,3);
					\draw (2,3)-- (0,2);
					\begin{scriptsize}
					\draw (0,2) circle (1.5pt);
					\draw (0,2.3) node {$2$};
					\draw (2,0) circle (1.5pt);
					\draw (2.0,-0.3) node {$5$};
					\draw (4,2) circle (1.5pt);
					\draw (4,2.3) node {$3$};
					\draw (4,0) circle (1.5pt);
					\draw (4,-0.3) node {$6$};
					\draw (0,0) circle (1.5pt);
					\draw (0,-0.3) node {$4$};
					\draw (2,3) circle (1.5pt);
					\draw (2.0,3.3) node {$1$};
					\end{scriptsize}
					\end{tikzpicture}
					\caption{}\label{cho}
				\end{center}
			\end{figure}    
		\end{example}
		Our aim is to classify  generalized block graphs whose binomial edge ideals admit 
		a unique extremal Betti number. Equivalently, we want to classify
		the generalized block graphs
		$G$ for which $\reg(S/J_G)=m(G)+1$. For that recall the definition of
		flower graph, introduced by Mascia and Rinaldo in \cite{CarlaR2018}: a \textit{flower} graph $F_{h,k}(v)$ is a connected 
		graph obtained  by gluing 
		each of $h$ copies of the complete graph $K_3$ and $k$ copies of the star graph $K_{1,3}$ at
		a common vertex $v$, that is free in each of them. 
		Now, we characterize generalized block graphs whose binomial edge ideals admit   a unique extremal Betti number.

		We denote by $\iv(G)$ the number of internal vertices of $G$.
		\begin{theorem}\label{unique}
			Let $G$ be a connected indecomposable generalized block graph. Then,  the following  are equivalent:
			\begin{enumerate}
				\item $S/J_G$ admits   a unique extremal Betti number.
				\item For any $v \in V(G)$, $F_{h,k}(v)$ is not an induced subgraph of $G$ for
				every $h, k \geq 0$ with $h + k \geq 3$.
			\end{enumerate}
			In this case, $\reg(S/J_G ) = m(G) + 1$.
		\end{theorem}
		\begin{proof}
			$(1) \implies (2):$ Suppose that for some $v \in V(G)$ and  $h,k \geq 0$ with $h+k \geq 3$,
			$F_{h,k}(v)$ is an induced subgraph of $G$. It 
			is enough to  prove  $\reg(S/J_G)>m(G)+1$. Let $H$ be an induced subgraph of $G$ obtained in the following way: 
			for every minimal cut set $A$ with $|A| \geq 2$, remove $|A|-1$ elements of $A$ from $G$. Note that $H$ is a block graph
			with  $\iv(H)=m(G)$ by \cite[Proposition 2.1]{Rauf} and $F_{h,k}(v)$ is an induced subgraph of $H$. 
			It follows from \cite[Theorem 8]{her2} that $\reg(S/J_H) >\iv(H)+1$. 
			Now, by virtue of \cite[Corollary 2.2]{MM}, $\reg(S/J_G)\geq \reg(S/J_H)>m(G)+1$. 
			
			$(2)\implies (1):$ By  Corollary \ref{reg-lb}, it is enough to prove  $\reg (S/J_G) \leq m(G)+1$.
			We prove this by induction
			on $m(G)$. If $m(G)=0$, then $G$ is a complete graph and the assertion is obvious. Assume 
			that $m(G) >0$. Let $F_1,\ldots, F_r$ 
			be a leaf order of $\mathcal{F}(\Delta(G))$.  Let $A$ be the minimal cut set defined in the
			proof of  Theorem \ref{main}. Then $G_A$, $G_A[\overline{A}]$
			and $G[\overline{A}]$ are  generalized block graphs.  Note that $G_A$ and $G_A[\overline{A}]$  are 
			generalized block graphs satisfying the hypothesis
			with $m(G_A)=m(G_A[\overline{A}])=m(G)-1$.  By induction, we have $\reg(S/J_{G_A})=\reg(S/J_{G_A[\overline{A}]})\leq m(G)$.
			As  in the proof
			of Theorem \ref{main}, $G[\overline{A}]$ has $q+1$ connected components, say  $H_1,\ldots,H_{q+1}$.
			Since  $G$ has no induced $F_{h,k}$ with $h+k \geq 3$, at least  $q-1$ components are isolated vertices.
			The two remaining components are a clique and a generalized block graph, say $H_1$, satisfying the assumption 
			with  $m(H_1) \leq m(G)-1$. Applying  induction  we
			obtain that $\reg(S/J_{G[\overline{A}]})\leq m(H_1)+2\leq m(G)+1$. Now, the assertion follows from the
			exact sequence (\ref{ses1}) and \cite[Corollary 18.7]{peeva}.
		\end{proof}
		As an immediate consequence of Proposition \ref{decomposable} and Theorem \ref{unique}, we have the following results:
		\begin{corollary}\label{least-reg}
			Let $G$ be a connected generalized block graph for which $G=G_1\cup \cdots \cup G_r$ is the decomposition
			of $G$ into indecomposable graphs. 
			Then,  $S/J_G$ admits   a unique extremal Betti number if and only if for each $i$, $S_i/J_{G_i}$ admits
			a unique extremal Betti number.
			Moreover, in this case $\reg(S/J_G) = m(G)+1$.
		\end{corollary}
		Recall that a {\it caterpillar} is a tree in which the removal of all pendant vertices leaves a path graph.
		\begin{corollary}\label{extremal-tree}
			Let $T$ be an indecomposable tree on $[n]$. Then,  $S/J_T$ admits   a unique extremal Betti
			number if and only if $T$ is a caterpillar.
		\end{corollary}
		The following example illustrates that the lower bound $m(G)+1 $ is not always attained.
		\begin{example}
			Let $G=F_{h,k}(v)$ be a flower graph with $h+k \geq 3$. Then, it follows from
			\cite[Corollary 3.5]{CarlaR2018} that $ \reg( S/J_G) = m(G)+h+k-1>m(G)+1.$
		\end{example}
		
		\section{Regularity upper bound for generalized block graph}
		In this section, we give an improved upper bound for the regularity of binomial
		edge ideals of generalized block graphs. Let $u,v \in V(G)$ be such that $e=\{u,v\} \notin E(G)$, 
		then we denote by $G_e$, the graph on the vertex set $V(G)$ and edge set 
		$E(G_e) = E(G) \cup \{\{x,y\} : x,\; y \in N_G(u) \; or \; x,\; y \in N_G(v) \}$. An edge $e$ is said to be a 
		\textit{cut edge} of $G$ if the number of connected components of $G\setminus e$ is  larger than  that of 
		$G$.

		We now recall a result from \cite{KM3} that will be used repeatedly in this section. 
		
		\begin{lemma}\label{cut-lemma}
			\cite[Proposition 2.1]{KM3} Let $G$ be a graph and $e$ be a cut edge of $G$. Then,
			$$\reg(S/J_G) \leq max\{\reg(S/J_{G\setminus e}), \reg(S/J_{(G\setminus e)_e})+1\}.$$
		\end{lemma}
		
		The \textit{degree} of a vertex  $v$ of $G$ is $\deg_G(v) =|N_G(v)|$. A vertex $v$ is said to
		be a \textit{pendant vertex} if $\deg_G(v) =1$.
		For $v \in V(G)$, let $\cdeg_G(v)$ denote the number of maximal cliques of $G$ which contains $v$, and $\pdeg_G(v)$
		denote the number of pendant vertices adjacent to $v$. Note that for every $v \in V(G)$, $\pdeg_G(v) \leq \cdeg_G(v)$.
		\begin{remark}\label{cdeg-rmk2}
			Let $G$ be a connected indecomposable generalized block graph which is not a star graph. 
			If $e=\{u,v\}$ is an edge with pendant vertex $u$,
			then $(G \setminus e)_e = (G\setminus u)_v \sqcup \{u\}$,
			$ \cl(G\setminus u) =\cl(G)-1$, $\cl((G\setminus u)_v) = \cl(G) -\cdeg_G(v) +1$,
			$J_{G\setminus e }= J_{G \setminus u}$ and $J_{(G \setminus e)_e} = J_{(G\setminus u)_v}$. 
			Also, $(G \setminus e)_e$ and $G\setminus e$ are generalized block graphs other than star graphs.
		\end{remark}
		
		A vertex $v \in V(G)$ with $\pdeg_G(v) \geq 1$ is said to be 
		of {\it type 1} if $\cdeg_G(v) = \pdeg_G(v)+1$, and of {\it type 2} if $ \cdeg_G(v) \geq \pdeg_G(v) +2$. We denote by 
		$\alpha(G)$, the number of 
		vertices of type $1$ in $G$  and by $\pv(G)$, the number of pendant vertices of $G$.
		
		\begin{lemma}
			If $G$ is a connected indecomposable graph on $[n]$ with $\pv(G)>0$, then $\pv(G)-\alpha(G)>0$.
		\end{lemma}
		\begin{proof} First, we assume that  $\alpha(G)=0$, then the claim follows. Now, assume that 
			$\alpha(G)=r>0$. Let $v_1,\ldots,v_r$ be all type $1$
			vertices of $G$. Since  $G$ is an indecomposable graph, $\cdeg_G(v_i)=\pdeg_G(v_i)+1 \geq 3$,
			for $i \in [r]$. Thus, $\pv(G)\geq 2 \alpha(G)$ which completes the proof.
		\end{proof}
		\begin{proposition}\label{tech-main}
			Let $G=G_1 \cup G_2$ be a decomposable  graph such that $G_1$ and $G_2$ are indecomposable
			and let $S_i={K}[x_j,y_j : j \in V(G_i)]$ for $i=1,2$. Suppose that one of the following conditions
			is satisfied: \begin{enumerate}
				\item[(a)] $G_1$ and $G_2$ are star graphs, or
				\item[(b)] $G_1$ is a star graph, $G_2$ is not a star graph and  
				$\reg(S_2/J_{G_2})\leq \cl(G_2)+\alpha(G_2)-\pv(G_2)$, or
				\item[(c)] for $i=1,2$, $G_i$ is not a star graph and $\reg(S_i/J_{G_i})\leq \cl(G_i)+\alpha(G_i)-\pv(G_i)$.
			\end{enumerate} Then  $\reg(S/J_G)\leq \cl(G)+\alpha(G)-\pv(G)$.
		\end{proposition}
		\begin{proof} (a) Since  $G_1$ and $G_2$ are indecomposable star graphs, $\alpha(G)=2$, $\pv(G)=\cl(G)-2$.
			By Proposition \ref{decomposable} and \cite[Theorem 4.1(a)]{MR3195706}, $\reg(S/J_G)=4=\cl(G)+\alpha(G)-\pv(G)$.
			Hence, the claim follows.
			\par (b) Let $V(G_1)\cap V(G_2)=\{u\}$. First, assume that $u$ is not a pendant vertex of $G_2$.
			Therefore, $\alpha(G)=\alpha(G_2)+1$ and $\pv(G)=\pv(G_1)+\pv(G_2)-1$. 
			Note that $\cl(G)=\cl(G_1)+\cl(G_2)$ and $\cl(G_1)=\pv(G_1)$. By \cite[Theorem 4.1(a)]{MR3195706}, 
			$\reg(S_1/J_{G_1})=2$. Hence,  Proposition \ref{decomposable} yields that 
			$\reg(S/J_G)\leq \cl(G)+\alpha(G)-\pv(G)$. We now assume that $u$ is a pendant vertex of $G_2$.
			Note that  $\pv(G)=\pv(G_1)+\pv(G_2)-2$ and $\cl(G_1)=\pv(G_1)$.
			Let $v \in N_{G_2}(u)$. If $v$ is of type $1$ in $G_2$, then $\alpha(G)= \alpha(G_2)$, and hence, 
			the claim follows from Proposition \ref{decomposable}.
			If $v$ is of type $2$ in $G_2$, then $\alpha(G)= \alpha(G_2)+1$, and hence,    
			by Proposition \ref{decomposable}, $\reg(S/J_G)\leq \cl(G)+\alpha(G)-\pv(G)$.
			\par (c) Let $V(G_1) \cap V(G_2) =\{u\}$. Observe that, $\cl(G) = \cl(G_1) +\cl(G_2)$. 
			If $\deg_{G_1}(u),\deg_{G_2}(u)>1$, then  $\alpha(G)=\alpha(G_1)+\alpha(G_2)$ and $\pv(G) = \pv(G_1)+\pv(G_2)$.
			Thus, by Proposition \ref{decomposable}, $\reg(S/J_G) =\reg(S_1/J_{G_1})+\reg(S_2/J_{G_2})\leq \cl(G)+\alpha(G) -\pv(G).$ 
			
			Assume that $u$ is a pendant vertex of $G_1$, let $N_{G_1}(u)=\{u_1\}$ and $\deg_{G_2}(u)>1$(or vice versa).
			Then $\pv(G) = \pv(G_1)+\pv(G_2) -1$. If $u_1$ is of type $1$ in $G_1$, then  we have 
			$\alpha(G) = \alpha(G_1) +\alpha(G_2) -1$.  If $u_1$ is of type $2$ in $G_1$, then we have  
			$\alpha(G) =\alpha(G_1)+\alpha(G_2)$. Hence, by Proposition \ref{decomposable}, 
			$\reg(S/J_G) = \reg(S/J_{G_1})+\reg(S/J_{G_2}) \leq \cl(G)+\alpha(G) -\pv(G).$
			
			Now, assume that for $i=1,2$, $u$ is a pendant vertex     of $G_i$, let $N_{G_i}(u)=\{u_i\}$. Then,
			$\pv(G) = \pv(G_1)+\pv(G_2) -2$. If both $u_1$ and $u_2$ are of type $1$ in $G_1$ and $G_2$, respectively,
			then   we have $\alpha(G) = \alpha(G_1) +\alpha(G_2) -2$.  If both $u_1$ and $u_2$ are of type 
			$2$ in $G_1$ and $G_2$, respectively, then $\alpha(G) =\alpha(G_1)+\alpha(G_2)$.
			If $u_1$ is of type $1$ in $G_1$ and $u_2$ is of type $2$ in $G_2$(or vice versa), then we have
			$\alpha(G) = \alpha(G_1) +\alpha(G_2) -1$.
			Hence,  Proposition \ref{decomposable} yields $\reg(S/J_G)=\reg(S_1/J_{G_1})+\reg(S_2/J_{G_2}) 
			\leq \cl(G)+\alpha(G) -\pv(G).$
		\end{proof}
		
		We now obtain a tight upper bound for the regularity of binomial edge ideals of connected indecomposable generalized block graphs. 
		\begin{theorem}\label{chordal-improved}
			Let $G$ be  a connected indecomposable  generalized block graph on $[n]$ which is not a star graph. Then,  
			$\reg(S/J_G) \leq \cl(G)+\alpha(G) - \pv(G)$.
		\end{theorem}
		\begin{proof}
			Let $k(G) =|\{ v : \pdeg_G(v) \geq 1 \}|$. We proceed by induction on $k(G)+m(G) \geq 0$. 
			For $k(G)=0$, $\pv(G)=\alpha(G)=0$, and hence, the result is immediate from \cite[Theorem 3.5]{MKM2018}
			or \cite[Theorem 3.15]{AR2}. If $m(G)=0$, then $G$ is a complete graph and the assertion is obvious. 
			Assume that $k=k(G) >0$, $m(G)>0$ and the assertion is true up to $k(G)+m(G) -1$.
			
			Let $v_1,\ldots,v_{k} \in V(G)$ be such that for each 
			$i=1,\ldots,k$, $\pdeg(v_i)=r_i \geq 1$. For each $i =1,\ldots,k$, let
			$e_{i,1} =\{v_i,w_{i,1}\},\ldots,e_{i,r_i}=\{v_i,w_{i,r_i}\}$ be pendant edges incident to $v_i$. 
			Since $G$ is an indecomposable graph, $\cdeg_G(v_i)=s_i \geq 3$.
			
			We proceed by induction on $r_{k}$. If $r_{k} =1$, then $v_{k}$ is of type $2$. Notice that  
			$k((G\setminus w_{k,1})_{v_k}) = k(G)-1$ and $m((G\setminus w_{k,1})_{v_k})=m(G)-1$. 
			Thus, by induction on $k(G)+m(G)$ and Remark \ref{cdeg-rmk2}, we have
			\begin{align*}
			\reg(S/J_{(G\setminus e_{k,1})_{e_{k,1}}}) & = \reg(S/J_{(G\setminus w_{k,1})_{v_k}}) \\
			& \leq \cl((G\setminus w_{k,1})_{v_k}) +\alpha((G\setminus  w_{k,1})_{v_k}) -\pv((G\setminus w_{k,1})_{v_k}) \\
			& = \cl(G) - \cdeg_G(v_k) + \alpha (G) -\pv(G)+2 \\ & \leq \cl(G)+\alpha(G)-\pv(G)-1.
			\end{align*}
			If $\cdeg_G(v_k)=3$, then $G\setminus w_{k,1} =G_1 \cup G_2$ is a decomposable graph. 
			If $G_i$ is not a star graph, then $k(G_i)\leq k(G)$ and $m(G_i)<m(G)$, and hence,
			by induction on $k(G)+m(G)$, we have $\reg(S_i/J_{G_i}) \leq \cl(G_i)+\alpha(G_i)-\pv(G_i)$.
			Note that  $G$ satisfies  the assumption of Proposition \ref{tech-main}.
			It follows from Proposition \ref{tech-main} that 
			\begin{align*}
			\reg(S/J_{G\setminus e_{k,1}}) & = \reg(S/J_{G\setminus w_{k,1}})\\
			&\leq \cl(G\setminus w_{k,1})+\alpha(G\setminus w_{k,1}) -\pv(G\setminus w_{k,1}) \\
			& = \cl(G) + \alpha(G) -\pv(G). 
			\end{align*}
			If $\cdeg_G(v_k)>3$, then $G\setminus w_{k,1}$ is an indecomposable generalized 
			block graph with $k(G\setminus w_{k,1})=k(G)-1$ and $m(G\setminus w_{k,1})=m(G)$. 
			Thus, by induction on $k(G)+m(G)$ and Remark \ref{cdeg-rmk2}, we have
			\begin{align*}
			\reg(S/J_{G\setminus e_{k,1}}) & = \reg(S/J_{G\setminus w_{k,1}}) \\ 
			& \leq \cl(G\setminus w_{k,1})+\alpha(G\setminus w_{k,1}) -\pv(G\setminus w_{k,1}) \\
			& = \cl(G) + \alpha(G) -\pv(G). 
			\end{align*}
			In both the cases, we get $\reg(S/J_{G\setminus e_{k,1}}) \leq \cl(G) + \alpha(G) -\pv(G).$
			Hence, by Lemma \ref{cut-lemma}, we have $$\reg(S/J_G) \leq \cl(G) + \alpha(G) -\pv(G).$$ 
			
			Assume now that $r_k >1$. By Remark \ref{cdeg-rmk2}, we get  $k((G \setminus w_{k,r_k})_{v_k})=k(G)-1$, 
			$\pv((G \setminus w_{k,r_k})_{v_k})=\pv(G)-r_k$,  $\cl((G \setminus w_{k,r_k})_{v_k})=\cl(G)-\cdeg_G(v_k)+1$ 
			and $m((G \setminus w_{k,r_k})_{v_k})<m(G)$. 
			
			\textbf{Case (1)}: $v_k$ is of type $2$ in $G$.
			
			Thus, $\alpha((G \setminus w_{k,r_k})_{v_k})=\alpha(G)$, $\cdeg_G(v_k)-r_k \geq 2$,  and hence,
			by Remark \ref{cdeg-rmk2} and induction on $k(G)+m(G)$, we have 
			\begin{align*}
			\reg(S/J_{(G\setminus e_{k,r_k})_{e_{k,r_k}}}) & = \reg(S/J_{(G \setminus w_{k,r_k})_{v_k}}) \\
			& \leq \cl((G \setminus w_{k,r_k})_{v_k}) +\alpha((G \setminus w_{k,r_k})_{v_k}) -\pv((G \setminus w_{k,r_k})_{v_k})\\
			& = \cl(G) -\cdeg_G(v_k) +1 + \alpha (G) -\pv(G) +r_k \\
			& \leq \cl(G) + \alpha(G) -\pv(G) -1.
			\end{align*}
			Note that $v_k$ is of type $2$ in $G\setminus w_{k,r_k}$ and $G\setminus w_{k,r_k}$ is indecomposable.
			It follows from Remark \ref{cdeg-rmk2}, and induction on $r_k$ that 
			\begin{align*}
			\reg(S/J_{G \setminus e_{k,r_k}}) &= \reg(S/J_{G\setminus w_{k,r_k}})\\
			& \leq \cl(G\setminus w_{k,r_k}) +\alpha(G\setminus w_{k,r_k}) - \pv(G\setminus w_{k,r_k}) \\
			& = \cl(G) + \alpha (G) - \pv(G).
			\end{align*}
			Thus, by Lemma \ref{cut-lemma}, $\reg(S/J_G) \leq \cl(G)+\alpha(G)-\pv(G)$.
			
			\textbf{Case (2)}: $v_k$ is of type $1$ in $G$.
			
			Thus, $\alpha((G \setminus w_{k,r_k})_{v_k})=\alpha(G)-1$, $\cdeg_G(v_k)-r_k =1$,  and hence, 
			by Remark \ref{cdeg-rmk2} and induction on $k(G)+m(G)$,
			\begin{align*}
			\reg(S/J_{(G\setminus e_{k,r_k})_{e_{k,r_k}}}) & = \reg(S/J_{(G \setminus w_{k,r_k})_{v_k}}) \\
			& \leq \cl((G \setminus w_{k,r_k})_{v_k}) +\alpha((G \setminus w_{k,r_k})_{v_k}) -\pv((G \setminus w_{k,r_k})_{v_k})\\
			& = \cl(G) -\cdeg_G(v_k) + \alpha (G) -\pv(G) +r_k \\
			& = \cl(G) + \alpha(G) -\pv(G) -1.
			\end{align*}
			Note that $v_k$ is of type $1$ in $G\setminus w_{k,r_k}$. 
			If $r_k=2$, then $G\setminus w_{k,r_k}$
			is a decomposable graph with decomposition $(G\setminus \{w_{k,1},w_{k,2}\}) \cup \{e_{k,1}\}$.
			Set $H=G\setminus \{w_{k,1},w_{k,2}\}$. If $H$ is not a star graph, then
			$k(H)\leq k(G)$ and $m(H)<m(G)$. Thus, by induction on $k(G)+m(G)$, $\reg(S/J_{H}) \leq \cl(H) + \alpha(H) -\pv(H).$ Note that            
			$\cl(H)=\cl(G)-2$. If $\deg_H(v_k)=1$, then $\pv(H)=p(G)-1$ and $\alpha(H) \leq \alpha(G)$, and if 
			$\deg_H(v_k)>1$, then  $\pv(H)=\pv(G)-2$ and
			$\alpha(H)=\alpha(G)-1$. Therefore, in both  cases, $\reg(S/J_{H}) \leq \cl(G) + \alpha(G) -\pv(G)-1$.
			Now, it follows from Proposition \ref{decomposable} that $\reg(S/J_{G\setminus e_{k,r_k}}) \leq \cl(G)+\alpha(G)-\pv(G)$. 
			If  $G\setminus \{w_{k,1},w_{k,2}\}$ is  a star graph, then $\alpha(G)=2$ and $\cl(G)-\pv(G)=1$.
			By Proposition \ref{decomposable},
			$\reg(S/J_{G\setminus e_{k,r_k}}) =3$.
			Thus,  $\reg(S/J_{G\setminus e_{k,r_k}}) \leq \cl(G)+\alpha(G)-\pv(G)$.
			
			If $r_k>2$, then $G\setminus w_{k,r_k}$ is indecomposable.  By induction on $r_k$, we have 
			\begin{align*}
			\reg(S/J_{G \setminus e_{k,r_k}}) &= \reg(S/J_{G\setminus w_{k,r_k}})\\
			& \leq \cl(G\setminus w_{k,r_k}) +\alpha(G\setminus w_{k,r_k}) - \pv(G\setminus w_{k,r_k}) \\
			& = \cl(G) + \alpha (G) - \pv(G)-1.
			\end{align*}  
			Thus, for $r_k \geq 2$, $\reg(S/J_{G \setminus e_{k,r_k}}) \leq
			\cl(G) + \alpha (G) - \pv(G).$ 
			Consequently, by   Lemma \ref{cut-lemma}, $\reg(S/J_G) \leq \cl(G)+\alpha(G) -\pv(G).$ Hence, the assertion follows.
		\end{proof}
		The following example illustrates that the upper bound is not always attained in Theorem \ref{chordal-improved}.
		\begin{example}
			Let $G$ be a tree as shown in Fig.  \ref{tree}. Notice that $\cl(G)=13, \; \alpha(G)=4$ and $\pv(G)=8$.
			Thus, $\cl(G)+\alpha(G)-\pv(G)=9$. 
			Using Macaulay2 \cite{M2}, we get $\reg(S/J_G)=8$.            
			\begin{figure}
				\begin{center}
					
					\begin{tikzpicture}[scale=.5]
					
					\draw  (1,-1)-- (3,-1);
					
					\draw  (3,-1)-- (5,-1);
					
					\draw  (5,-1)-- (5,1);
					
					\draw  (5,1)-- (7,1);
					
					\draw  (5,-1)-- (7,-1);
					
					\draw  (7,-1)-- (9,-1);
					
					\draw  (9,-1)-- (9,1);
					
					\draw  (9,-1)-- (11,-1);
					
					\draw  (7,-1)-- (7,-3);
					
					\draw  (7,-3)-- (5,-3);
					
					\draw  (7,-3)-- (9,-3);
					
					\draw  (3,-3)-- (3,-1);
					
					\draw  (3,-1)-- (3,-3);
					
					\draw  (5,1)-- (3,1);
					
					\begin{scriptsize}
					
					\draw  (1,-1) circle (1.5pt);
					
					\draw (1,-0.6) node {$1$};
					
					\draw  (3,-1) circle (1.5pt);
					
					\draw (3,-0.6) node {$2$};
					
					\draw  (5,-1) circle (1.5pt);
					
					\draw (5.3,-0.6) node {$4$};
					
					\draw  (5,1) circle (1.5pt);
					
					\draw (5,1.4) node {$5$};
					
					\draw  (7,1) circle (1.5pt);
					
					\draw (7,1.4) node {$7$};
					
					\draw  (7,-1) circle (1.5pt);
					
					\draw (7,-0.6) node {$8$};
					
					\draw  (9,-1) circle (1.5pt);
					
					\draw (9.3,-0.6) node {$9$};
					
					\draw  (9,1) circle (1.5pt);
					
					\draw (9,1.4) node {$10$};
					
					\draw (11,-1) circle (1.5pt);
					
					\draw (11,-0.6) node {$11$};
					
					\draw  (7,-3) circle (1.5pt);
					
					\draw (7,-3.4) node {$12$};
					
					\draw  (5,-3) circle (1.5pt);
					
					\draw (5,-3.4) node {$13$};
					
					\draw  (9,-3) circle (1.5pt);
					
					\draw (9,-3.4) node {$14$};
					
					\draw  (3,-3) circle (1.5pt);
					
					\draw (3,-3.4) node {$3$};
					
					\draw  (3,1) circle (1.5pt);
					
					\draw (3,1.4) node {$6$};
					
					\end{scriptsize}
					
					\end{tikzpicture}
					\caption{}\label{tree}
				\end{center}            
			\end{figure}
		\end{example}
		As an immediate consequence, we have the following:
		\begin{corollary}\label{reg-tree}
			Let $T$ be an indecomposable tree on $[n]$ which is not 
			a star graph. Then,  $\reg(S/J_T) \leq \cl(T) + \alpha(T) -\pv(T)$.
		\end{corollary}
		Finally we show two classes of block graphs that attain
		the upper bound  $\cl(G)+\alpha(G) -\pv(G).$
		\begin{corollary}
			If $G$ is an indecomposable caterpillar that is not a star graph or a flower graph, 
			then $\reg(S/J_G)=\cl(G)+\alpha(G) -\pv(G).$
		\end{corollary}
		\begin{proof}
			First, let $G$ be an indecomposable caterpillar that is not a star graph. Then 
			$m(G)+1= \cl(G)+\alpha(G) -\pv(G)$. Therefore, by  Corollaries \ref{extremal-tree} 
			and \ref{reg-tree},  $\reg(S/J_G)= \cl(G) + \alpha(G) -\pv(G)$.
			
			Now, let $G=F_{h,k}(v)$ be a flower graph. Then $\cl(G)+\alpha(G)-\pv(G)=m(G)+\cdeg_G(v)-1=h+2k$. 
			Thus, by virtue of \cite[Corollary 3.5]{CarlaR2018}, $\reg(S/J_G)=\cl(G)+\alpha(G)-\pv(G)$.
			
		\end{proof}

\vskip 2mm
\noindent
\textbf{Acknowledgements:} The author is grateful to his advisor, Prof. A. V. Jayanthan for the constant support, valuable ideas, and suggestions.
The author thanks the National Board for Higher Mathematics, India, for the financial support. The author also wishes to express his sincere gratitude to the anonymous referee whose comments help to improve the exposition in great detail.

\bibliographystyle{plain}  
\bibliography{biblog}
\end{document}